\documentclass{amsart}
\usepackage[english ]{babel}
\usepackage{amsmath,latexsym,amssymb,verbatim}
\usepackage{pdfsync}

\usepackage[all]{xy}

\usepackage{hyperref}

\newtheorem{theorem}{Theorem}[section]
\newtheorem{lemma}[theorem]{Lemma}
\newtheorem{proposition}[theorem]{Proposition}
\newtheorem{corollary}[theorem]{Corollary}
\newtheorem{remark}[theorem]{Remark}
\newtheorem{note}[theorem]{Note}

\newtheorem{Formula of adjoint functors}[theorem]{Formula of adjoint functors}

\newtheorem{definition}[theorem]{Definition}
\newtheorem{notation}[theorem]{Notation}

\newtheorem{Adjunction formula}[theorem]{\indent\sc Adjunction formula}
\newtheorem{hypothesis}[theorem]{Hypothesis}

\DeclareMathOperator{\limi}{{lim}}
\newcommand{\ilim}[1]{\,\underset{#1}{\underset{\to}{\limi}}\,}

\DeclareMathOperator{\Id}{{Id}}
\DeclareMathOperator{\Hom}{{Hom}}

\DeclareMathOperator{\Ker}{{Ker}}
\DeclareMathOperator{\Ima}{{Im}}
\DeclareMathOperator{\ZZ}{{\mathbb Z}}
\DeclareMathOperator{\R}{{\mathcal R}}
\newcommand{\dosflechasa}[3][]{\xymatrix@1{\ar@<1ex>[r]^-{#2}
\ar@<-1ex>[r]_-{#3} & }}
\newcommand{\dosflechas}{{\xymatrix@1  {\ar@<1ex>[r]
\ar@<-1ex>[r] & }}}
\newcommand{\dosflechasab}[3][]{\xymatrix@1  @C10pt {\ar@<1ex>[r]^-{#2}
\ar@<-1ex>[r]_-{#3} & }}
\newcommand{\dosflechasb}{{\xymatrix@1 @C10pt {\ar@<1ex>[r]
\ar@<-1ex>[r] & }}}

\NoCompileMatrices

\input arrow.tex

\begin{document}

\title{Functors of modules associated with flat and projective modules II}

\author{Adri\'an Gordillo-Merino, Jos\'e Navarro, Pedro Sancho}
\address{Departamento de Matem\'aticas\\
Universidad de Extremadura\\
Avenida de Elvas, s/n\\
06006 Badajoz (SPAIN)} 
\email{adgormer@unex.es, navarrogarmendia@unex.es, sancho@unex.es}

\thanks{All authors have been partially supported by Junta de Extremadura and FEDER funds.}


\begin{abstract} Let $R$ be an associative ring with unit. Given an $R$-module $M$, we can associate the following covariant functor from the category of $R$-algebras to the category of abelian groups: $S\mapsto M\otimes_R S$. 
With the corresponding notion of dual functor, we prove that the natural morphism of functors $\,\mathcal M\to \mathcal M^{\vee\vee}\,$ is an isomorphism. 
We prove several characterizations of the functors associated with  flat modules, flat Mittag-Leffler modules and projective modules.
\end{abstract}

\maketitle

\section{Introduction}

Let $\,R\,$ be an associative ring with  unit. Consider the functor from the category of $R$-algebras $\,R\text{-Alg}\,$ to  the category of right $R$-modules $\,R\text{-Mod}\,$,
$$o\colon \,R\text{-Alg}\,\to \,R\text{-Mod}\,,\, o(S):=S$$  
for any $R$-algebra $S$, and the functor $r\colon \,R\text{-Mod}\,\to \,R\text{-Alg}\,$, $r(N)=R\langle N\rangle$, where $R\langle N\rangle$ is the $R$-algebra generated by $N$ (see \ref{N3.4}).
It is well known the functorial isomorphism
$$\Hom_R(N,o(S))=\Hom_{R-alg}(r(N),S)$$
for any right $R$-module $N$ and any $R$-algebra $S$.
Hence, it is easy to obtain a functorial isomorphism
$$\Hom_{grp}(\mathbb G\circ o,\mathbb F)=\Hom_{grp}(\mathbb G,\mathbb F\circ r)$$
for any covariant functors of abelian groups $\mathbb G\colon \,R\text{-Mod}\,\to $ $\mathbb Z${-Mod},  and $\mathbb F\colon \,R\text{-Alg }\,\to $ $\mathbb Z${-Mod}.

Let $\,\mathcal R\,$ be the covariant functor from the category of $\,R$-algebras,  to the category of $R$-algebras, defined by $\,{\mathcal R}(S):=S$, for any $R$-algebra $\,S$. 

\medskip
\begin{definition}
A {\sl functor of $\,\mathcal{R}$-modules} is a covariant functor  $\,\mathbb M \colon R\text{-Alg} \to \mathbb Z\text{-Mod} \,$ together with a morphism of functors of sets $\,{\mathcal R}\times \mathbb M\to \mathbb M\,$ that endows
$\,\mathbb M(S)\,$ with an $\,S$-module structure, for any $\,R$-algebra $\,S$. 

A {\sl morphism of $\,{\mathcal R}$-modules} $\,f\colon \mathbb M\to \mathbb M'\,$
is a morphism of functors such that the morphisms $\,f_{S}\colon \mathbb M({S})\to
\mathbb M'({S})\,$ are morphisms of $\,{S}$-modules.
\end{definition}
\medskip

If $\mathbb G\colon \,R\text{-Mod}\,\to $ $\mathbb Z${-Mod} is  additive, then $\mathbb G^o:=\mathbb G\circ o$ is naturally 
a functor of $\R$-modules. Let $in,h_x\colon R\langle N\rangle \to  R\langle N\oplus x\cdot R\rangle $ be the morphisms of $R$-algebras induced by the morphisms of $R$-modules $N\to  R\langle N\oplus x\cdot R\rangle $, $n\mapsto n,x\cdot n$. Given a functor of $\R$-modules $\mathbb F$, let $\mathbb F^r\colon \,R\text{-Mod}\,\to $ $\mathbb Z${-Mod} be  defined as follows: $\mathbb F^r(N)$ is the kernel of the morphism 
$$\xymatrix{\mathbb F(R\langle N\rangle) \ar[rr]^-{\mathbb F(h_x)-{x\cdot \mathbb F(in)}}  & & \mathbb  F(R\langle N\oplus x\cdot R\rangle)\\ n' \ar@{|->}[rr] & & \mathbb F(h_x)(n')- x\cdot \mathbb F(in)(n'),}$$
for any right $R$-module $N$ and any $n'\in  \mathbb F(R\langle N\rangle)$.
We prove (\ref{T3}) the following theorem.

\begin{theorem} \label{T4} Let $\mathbb F\colon \,R\text{-Alg}\,\to $ $\mathbb Z${-Mod}  be a covariant functor of $\R$-modules and $\mathbb G\colon \,R\text{-Mod}\,\to $ $\mathbb Z${-Mod}  an additive covariant functor of abelian groups.  Then, we have a functorial isomorphism
$$\xymatrix  { \Hom_{grp}(\mathbb G,\mathbb F^r) \ar@{=}[r] & 
\Hom_{\R}(\mathbb G^o,\mathbb F)}$$
\end{theorem}

In Algebraic Geometry, functors of $\R-$modules from the category of $R-$algebras to the category of abelian groups are featured more frequently than those from the category of right $R-$modules to the category of abelian groups. Nevertheless,  the results 
about reflexivity of modules and characterizations of flat Mittag-Leffler modules are obtained more naturally with the latter functors (see
 \cite{mittagleffler}). The results below are a consequence of Theorem \ref{T4} and the results obtained in  \cite{mittagleffler}.

Any $R$-module $M$ can be thought as a functor of $\R$-modules:
Consider the following covariant functor of $\,\mathcal R$-modules $\mathcal M$, defined by 
$$\mathcal M(S):=S\otimes_R M\,,$$
for any $R$-algebra $S$. We will say that $\mathcal M$ is the quasi-coherent $\R$-module associated with  $M$. It is easy to prove that the category of $R$-modules is equivalent to the category of quasi-coherent $\R$-modules. 

Given a functor of $\R$-modules 
$\mathbb M$,  we will say that the functor of right $\R$-modules $\mathbb M^\vee$ defined by
$$\mathbb M^\vee(S)=\Hom_{\R}(\mathbb M,\mathcal S),$$
for any $R$-algebra $S$, is the dual functor of $\mathbb M$.
We will say that $\mathcal M^\vee$ is the $\R$-module scheme associated with the $R$-module $M$.
We are now in a position to state the main results of this paper, grouped together herein according to the concepts involved:

\begin{theorem} Let $M$ be an $R$-module. The natural morphism of $\mathcal R$-modules $$\mathcal M \to \mathcal M^{\vee\vee}$$
is an isomorphism.
\end{theorem}

When $\,R\,$ is a commutative ring,  this theorem has been proved for finitely generated modules 
using the language of sheaves in the big Zariski topology, in  \cite{Hirschowitz}, and it is
implicit in \cite[II,\textsection 1,2.5]{gabriel}. The reflexivity of these quasi-coherent $\,\mathcal{R}$-modules $\,\mathcal{M}\,$ has been used for a variety of applications in theory of linear representations of affine group schemes \cite{Amel,Pedro1,Pedro2}.
Likewise, we think that this new reflexivity theorem will be useful in the theory of comodules over non-commutative rings.

\bigskip

\begin{theorem} Let $M$ be an $R$-module such that $\mathcal M^r(N)=N\otimes_R M$, for any right $R$-module $N$. Then,
\begin{enumerate}
\item  $M$ is a finitely generated projective module if and only if $\mathcal M$ is a module scheme.

\item $M$ is a flat module if and only if $\mathcal M$ is a direct limit of module schemes.

\item $M$ is a flat Mittag-Leffler module if and only if  $\mathcal M$ is a direct limit of submodule schemes.

\item $M$ is a flat strict Mittag-Leffler module if and only if  $\mathcal M$ is a direct limit of submodule schemes, $\mathcal N_i^\vee\subseteq \mathcal M$, and the dual morphism $\mathcal M^\vee\to\mathcal N_i$ is an epimorphism, for any $i$.

\item $M$ is a countably generated projective module  if and only if there exists a chain of 
module subschemes of $\mathcal M$, 
$$ \mathcal N_0^\vee\subseteq \mathcal N_1^\vee\subseteq\cdots\subseteq \mathcal N^\vee_n\subseteq\cdots,$$
such that $\mathcal M=\cup_{n\in\mathbb N}\mathcal N_n^\vee$. 

\item $M$ is projective if and only if there exists a chain of $\R$-sub\-mo\-du\-les of $\mathcal M$,
 $$\mathbb W_0\subseteq \mathbb W_1\subseteq \cdots\subseteq \mathbb W_n\subseteq \cdots ,$$
such that $\mathcal M=\cup_{n\in\mathbb N} \mathbb W_n$, where $\mathbb W_n$ is a direct sum of module schemes  and the natural morphism $\mathcal M^\vee\to \mathbb W_{n}^\vee$ is an epimorphism, for any $n\in \mathbb N$.

\end{enumerate}

\end{theorem}

Also, the theorem below establishes several statements which are equivalent to an $R$-module being a flat strict Mittag-Leffler module.

\begin{theorem} Let $M$ be an $R$-module such that $\mathcal M^r(N)=N\otimes_R M$, for any right $R$-module $N$.
Then, the following, equivalent  conditions are met:

\begin{enumerate}

\item $M$ is  a flat strict Mittag-Leffler $R$-module.

\item Let $\{M_i\}_{i\in I}$ be the set of all finitely generated submodules of $M$, and $M'_i:=\Ima[M^*\to M_i^*]$. The natural morphism
$\mathcal M\to \lim \limits_{\rightarrow} {\mathcal M_i'}^\vee$
is an isomorphism.

\item There exists a monomorphism $\mathcal M\hookrightarrow \prod_{J}\mathcal R$.

\item Every morphism of $\mathcal R$-modules $f\colon \mathcal M^\vee\to \mathcal R$ factors through the quasi-coherent module associated with $\Ima f_R$.
\end{enumerate}

\end{theorem}

\section{Preliminaries}\label{preliminar}

Let $\,R\,$ be an associative ring with  unit, and let $\,\mathcal R\,$ be the covariant functor from the category of $\,R$-algebras to the category of $\, R$-algebras, defined by $\,{\mathcal R}(S):=S$, for any $R$-algebra $\,S$.


\medskip
\begin{definition}
A {\sl functor of $\,\mathcal{R}$-modules} is a covariant functor  $\,\mathbb M \colon R\text{-Alg} \to \mathbb Z\text{-Mod}\,$ together with a morphism of functors of sets $\,{\mathcal R}\times \mathbb M\to \mathbb M\,$ that endows
$\,\mathbb M(S)\,$ with an $\,S$-module structure, for any $\,R$-algebra $\,S$. 

A {\sl morphism of $\,{\mathcal R}$-modules} $\,f\colon \mathbb M\to \mathbb M'\,$
is a morphism of functors such that the morphisms $\,f_{S}\colon \mathbb M({S})\to
\mathbb M'({S})\,$ are morphisms of $\,{S}$-modules.
\end{definition}
\medskip

\begin{definition}  If $\,\mathbb M\,$ is an $\,\mathcal R$-module, the {\sl  dual} $\,\mathbb M^\vee\,$ is the following functor $\, R\text{-Alg} \to \mathbb Z\text{-Mod}\,$ 
$$\mathbb M^\vee (S):=
\Hom_{\mathcal R}(\mathbb M,\mathcal S),$$
which is a functor of right $\R$-modules.\end{definition}

If $\,S\,$ is an $R$-algebra, the restriction of an $\,{\mathcal R}$-module $\,\mathbb M\,$ to the category of ${S}$-algebras will be written $$\mathbb M_{\mid {S}}(S'):=\mathbb M(S'),$$ for any ${S}$-algebra $S'$.

\medskip
\begin{definition}
The {\sl functor of homomorphisms} $\,{\mathbb Hom}_{{\mathcal R}}(\mathbb M,\mathbb M')\,$ is the covariant functor $\, R\text{-Alg} \to \mathbb Z\text{-Mod}\,$ defined by $${\mathbb Hom}_{{\mathcal R}}(\mathbb M,\mathbb M')({S}):={\rm Hom}_{\mathcal {S}}(\mathbb M_{|{S}}, \mathbb M'_{|{S}}), $$ where $\,\Hom_{{\mathcal S}}(\mathbb M_{|{S}},\mathbb M'_{|{S}})$ stands for the set\footnote{In this paper, we will only consider well-defined functors $\,{\mathbb Hom}_{{\mathcal R}}(\mathbb M,\mathbb M')$, that is to say, functors such that $\,\Hom_{\mathcal {S}}(\mathbb M_{|{S}},\mathbb {M'}_{|{S}})\,$ is a set, for any $R$-algebra ${S}$.} of all morphisms of $\,{\mathcal S}$-modules from $\,\mathbb M_{|{S}}\,$ to $\,\mathbb M'_{|{S'}}$.

In the following, it will also be convenient to consider another notion of dual module:
$\,\mathbb{M}^*\,$ is the functor  of right $\R$-modules defined by $$\,\mathbb M^*:=\mathbb Hom_{{\mathcal R}}(\mathbb M,\mathcal R). $$
\end{definition}
\medskip



\medskip
\begin{definition} The {\sl quasi-coherent $\mathcal{R}$-module} associated with an $R$-module $\,M\,$ is the following covariant functor 
$$\,{\mathcal{M}} \colon R\text{-Alg} \to {\mathbb Z}\text{-Mod}\, , \quad \mathcal{M} (S) := S \otimes_R M. $$
\end{definition}
\medskip


Quasi-coherent modules are determined by its global sections. In particular, we will make use of the following statement, whose proof is immediate:

\medskip
\begin{proposition} \label{tercer}
Restriction to global sections $\,f\mapsto f_R\,$ defines a bijection:
$${\rm Hom}_{\mathcal R} ({\mathcal M}, \mathbb M) = {\rm Hom}_R (M, \mathbb M(R))\,, $$ for any quasi-coherent $\,\mathcal{R}$-module $\,\mathcal{M}\,$ and any $\,{\mathcal R}$-module $\,\mathbb M$.
\end{proposition}
\medskip

As a consequence, both notions of dual module introduced above coincide on quasi-coherent modules; that is, $\,\mathcal M^*=\mathcal M^\vee\,$. 

In fact, if $\,S\,$ is an $R$-algebra, then 
$$\,\mathcal{M}^\vee (S) = \Hom_{\mathcal R}(\mathcal{M},\mathcal S) = \Hom_{R}(M , S)\,$$
and, as $\,{\mathcal M}_{\mid {S}}\,$ is the quasi-coherent $\mathcal {S}$-module associated with $\,S\otimes_R M \,$, 
$$\mathcal M^*(S)=\Hom_{\mathcal S}(\mathcal M_{|S},\mathcal S)=\Hom_S(S\otimes_R M,S)=\Hom_R(M,S) = \mathcal{M}^\vee (S) \ . $$




\bigskip

\begin{definition} We will say that a functor from the category of right  $R$-modules to the category of abelian groups is a functor of abelian groups.

\end{definition}

\begin{definition} Given  a functor of abelian groups $\mathbb G$, let $\mathbb G^o$ be the functor from the category of $R$-algebras to the category of abelian groups defined by 
$$\mathbb G^o(S):=\mathbb G(S),$$
for any $R$-algebra $S$. Given a morphism of $R$-algebras $w\colon S\to S'$
then $\mathbb G^o(w):=\mathbb G(w)$.

\end{definition}

\begin{remark} \label{recall}
Observe that we can define
$s* g:=\mathbb G(s\cdot)(g)$, for any
$g\in \mathbb G^o(S)$ and $s\in S$ (where $s\cdot S\to S$ is defined by $s\cdot (s'):=s\cdot s'$). If $\mathbb G$ is additive, then $\mathbb G^o$ is a functor of $\R$-modules.
\end{remark}

Any morphism $\phi\colon \mathbb G\to \mathbb G'$  of functors of abelian groups
defines the morphism $\phi^o\colon \mathbb G^o\to \mathbb G'^o$, $\phi^o_S:=\phi_S$ for any $R$-algebra $S$. Obviously,
$$\phi^o_S(s* g)=\phi_S(\mathbb G(s\cdot)(g))=\mathbb G'(s\cdot )(\phi_S(g))=
s*\phi^o_S(g).$$

Finally, any definition or statement in the category of $\,\mathcal R$-modules has a corresponding definition or statement in the category of right $\,\mathcal R$-modules, that we will use without more explicit mention.

As examples, if $\,\mathbb{M}\,$ is an $\,\mathcal{R}$-module, then $\,\mathbb M^* = {\mathbb Hom}_{\mathcal R}(\mathbb M,\mathcal R)\,$ is a right $\,\mathcal R$-module. If $\,\mathbb N\,$ is a right $\,\mathcal R$-module, then the dual module defined by 
$$\mathbb N^*:=\mathbb Hom_{\mathcal R}(\mathbb N,\mathcal R)$$ is an $\,\mathcal R$-module, etc.

\section{Extension of a functor on the category of algebras to a functor on the category of modules}

\label{SectExt}

\begin{notation} 
If $\,M\,$ is an $\,R$-module, observe that $\,M\otimes_{\ZZ} R\,$ is an $R$-bimodule and we can 
consider the tensorial $\,R$-algebra \label{N3.4}
$$R\langle M\rangle :=T^\cdot_{R} (M\otimes_{\ZZ} R)=(T^{\cdot}_{\ZZ} M)\otimes_{\ZZ} R \, . $$
\end{notation}
\medskip

\begin{remark} If $N$ is a right $R$-module, then:
$$R\langle N\rangle:=T^\cdot_R(R\otimes_{\ZZ} N) \, . $$ 
\end{remark}
\medskip

\begin{lemma} The following functorial map is bijective:
$$\Hom_{R-alg}(R\langle M\rangle, S)\to \Hom_{R}(M,S) \ , \quad f\mapsto f'\  ,$$
where $\,f'(m):=f(m\otimes 1)\,$ for any $\, m\in M\,$.\end{lemma}

\begin{proof} $ \Hom_{R-alg}( T^\cdot_{R} (M\otimes_{\ZZ} R),S)=\Hom_{R\otimes_{\ZZ} R} (M\otimes_{\ZZ} R,S) =\Hom_R(M,S).$
\end{proof} 
\medskip 
 
 Any $\,R$-linear morphism $\,\phi\colon M\to M'\,$ uniquely extends to a morphism of $\,R$-algebras $\,\tilde \phi\colon R\langle M\rangle \to R\langle M'\rangle$, $\,m\otimes 1\mapsto \phi(m)\otimes 1$.

If we use the notation 
$$M\overset {n}\cdots M\cdot R:=M\otimes_{\ZZ}\overset {n}\cdots\otimes_{\ZZ} M\otimes_{\ZZ} R\,,\,\,\, m_1\cdots m_{n}\cdot r\mapsto m_1\otimes\cdots\otimes m_n\otimes r,$$ 
then
$$R\langle M\rangle =\oplus_{n=0}^\infty \,M\overset {n}\cdots M\cdot R\,,$$ and the product in this algebra can be written as follows:
$$(m_1\cdots m_{n}\cdot r)\cdot (m'_1\cdots m'_{n'}\cdot r')=
m_1\cdots m_{n}\cdot (r m'_1)\cdot m'_2\cdots m'_{n'}\cdot r'.$$

\medskip

\medskip
\begin{notation} Let us use the following notation 
$$M\oplus Rx:=M\oplus R \quad , \quad (m,r\cdot x)\mapsto (m,r) \, . $$
Likewise, if $M$ is a right $R$-module $M\oplus xR:=M\oplus R,$ $(m,x\cdot r)\mapsto (m,r) \, . $
\end{notation}
\medskip

\begin{notation} Let $M$ be a right $R-$module. Consider the morphisms of $R-$algebras $$in,h_x: R\langle M\rangle \longrightarrow R\langle M\oplus xR\rangle $$ induced  by the morphism of $R$-modules $M\to R\langle M\oplus xR\rangle$, $m\mapsto m,\,x\cdot m$ for any $m\in M\,$. 
\end{notation}

\medskip
\begin{definition} \label{D4.2} Given a functor  of  $\,\R -$modules $\,\mathbb F\,$, let $\,{\mathbb F}^r\,,$ be the functor of abelian groups, defined as follows: $\,{\mathbb F}^r(M)\,$
is the kernel of the morphism
$$\xymatrix{\mathbb F(R\langle M\rangle) \ar[rr]^-{\mathbb F(h_x)-x\cdot\mathbb F(in)} &   &\mathbb  F(R\langle M\oplus xR\rangle)\\ f \ar@{|->}[rr] &  & \mathbb F(h_x)(f)- x\cdot \mathbb F(in)(f),}$$
for any right $R$-module $M$ and any $f\in  \mathbb F(R\langle M\rangle)$.
\end{definition}

If $\,w\colon M\to M'\,$ is a morphism of $R$-modules, it induces morphisms of $R$-algebras
$$R\langle w\rangle \colon R\langle M\rangle 
\to R\langle M'\rangle \ ,  \quad R\langle w\rangle(m)=w(m)$$ and 
$R\langle w\oplus 1\rangle\colon R\langle M\oplus xR\rangle 
\to R\langle M'\oplus xR\rangle$, $R\langle w\oplus 1\rangle(m)=w(m)$, $R\langle w\oplus 1\rangle(x)=x$.
Observe that $R\langle w\oplus 1\rangle \circ h_x=h_x\circ R\langle w\rangle$. Hence,
we have the morphism 
$${\mathbb F}^r(w)\colon {\mathbb F}^r(M)\to {\mathbb F}^r(M')\, , \quad {\mathbb F}^r(w)(f):=\mathbb F(R\langle w\rangle)(f)$$ for  any $f\in {\mathbb F}^r(M)\subset \mathbb F(R\langle M\rangle)$.

\begin{note} In a similar vein, we can define the {\sl extension} of a functor $\,\mathbb F\,$ of right $\R$-modules, which is a functor $\,{\mathbb F}^r\,$ from the category of $R$-modules to the category of abelian groups.\end{note}

\begin{notation} 
Let $N$ be a right  $R$-module and let $M$ be an $R$-module.  Consider the sequence of morphisms of groups
$$N\otimes_RM\overset{i}\to N\otimes_R M\otimes_{\ZZ} R\dosflechasa{p_1}{p_2}   N\otimes_R M\otimes_{\ZZ} R\otimes_{\ZZ} R$$
where $\,i(n\otimes m):=n\otimes m\otimes 1$, $\,p_1(n\otimes m\otimes r):=n\otimes m\otimes r\otimes 1$ and
$\,p_2(n\otimes m\otimes r):=n\otimes m\otimes 1\otimes r$, is exact.
\end{notation}

\begin{lemma} \label{reperab} Let $\,M\,$ be an $\,R$-module and $\,N\,$ a right $\,R$-module. Then, 
$$\mathcal N^r(M)=\Ker[N\otimes_R M\otimes_{\ZZ} R\dosflechasab{p_1}{p_2} N\otimes_R   M\otimes_{\ZZ} R\otimes_{\ZZ} R],$$
\end{lemma}

\begin{proof}  It is easy to prove that the kernel of the morphism
$$N\otimes_R R\langle M\rangle\to N\otimes_R  R\langle M\rangle[x] , \,\,n\otimes p(m)\mapsto
n\otimes (p(m)x-p(mx))$$
is included in  $N\otimes_R M\otimes_{\ZZ} R$. 
Observe that the morphism of $R$-algebras $R\langle M\oplus Rx\rangle\to R\langle M\rangle[x]$, $m\mapsto m$ and $x\mapsto x$, is an epimorphism. 

Then,  $$\mathcal N^r(M)\subseteq  N\otimes_R M\otimes_{\ZZ} R$$ and 
$\mathcal N^r(M)=\Ker(p_1-p_2).$
\end{proof}

\medskip
\begin{remark} \label{repera2b} Observe that 
$${\mathcal N}^r(M)=\Ker[N\otimes_R M\otimes_{\ZZ} R\overset{p_1-p_2}\longrightarrow N\otimes_R M\otimes_{\ZZ} R\otimes_{\ZZ} R]={\mathcal M}^r(N)\ . $$
\end{remark}
\medskip

\begin{proposition} \label{super} Let $\,N\,$ be a right $\,R$-module and let $\,M\,$ be an $\,R$-module. If  $\,M\,$ (or $\,N\,$) is an $\,R$-bimodule or a flat module, then $$\mathcal N^r(M)=N\otimes_R M.$$
\end{proposition}

\begin{proof} By Lemma \ref{reperab}, we have to prove that $\Ker(p_1-p_2)=N\otimes_R M$.

Suppose that $\,M\,$ is a bimodule.
It is clear that $\,\Ima i\subseteq \Ker(p_1-p_2)$. 
Let 
$\,s\colon N\otimes_R M\otimes_{\ZZ} R\to N\otimes_R M $, $s(n\otimes m\otimes r)=n\otimes mr$
and $$\,s'\colon N\otimes_R M\otimes_{\ZZ} R \otimes_{\ZZ} R \to N\otimes_R M \otimes_{\ZZ} R\,,\,\,\, s'(n\otimes m\otimes r\otimes r')=n\otimes mr\otimes r'.$$ 
Observe that $\,s\circ i=\Id$, so that $i$ is injective. Also, $\,s'\circ p_2=\Id\,$ and $\,s'\circ p_1=i\circ s$. Thus, if $\,x\in \Ker(p_1-p_2)$, then $\,p_1(x)-p_2(x)=0$; hence, $\,0=s'(p_1(x))-s'(p_2(x))=i(s(x))-x\,$ and $\,x\in \Ima i$.
Then,  $\Ker(p_1-p_2)=N\otimes_R M$.

In particular, taking the bimodule $\,M=R\,$, the following sequence of morphisms of groups is exact:
$$N\overset{i}\to N\otimes_{\ZZ} R\dosflechasa{p_1}{p_2}   N\otimes_{\ZZ} R\otimes_{\ZZ} R \ . $$ 
Thus, if $\,M\,$ is flat, tensoring by $\,M\,$ it  also follows that  $\Ker(p_1-p_2)=N\otimes_R M$.

\end{proof}

\begin{proposition} \label{super2} If there exists a central subalgebra $\,R'\subseteq R\,$ such that $\,Q\to Q\otimes_{R'} R\,$ is injective, for any $\,R'$-module $\,Q\,$, then
$$\mathcal N^r(M)=N\otimes_R M.$$

\end{proposition}

\begin{proof}
Let us write $\,M':=M\otimes_{R'} R\,$, which is a bimodule as follows:
$$r_1\cdot (m\otimes  r)\cdot r_2=r_1m\otimes rr_2.$$
The morphism of $R$-modules $i\colon M\to M'$, $i(m):= m\otimes 1$ is universally injective: Given an $R$-module $P$, put $Q:=P\otimes_R M$. Then, the morphism $P\otimes_R M=Q\to Q\otimes_{R'} R=P\otimes_R M'$ is injective.

Put  $Q:=M'/M$ and $M'':=Q\otimes_{R'} R$.  Let $p_1$ be the composite morphism
 $M'\to M'/M=Q\to Q\otimes_{R'}R =M''$.
The sequence of morphisms of $R$-modules
$$0\to M\overset i\to M'\overset p\to M''$$
is universally exact. Consider the following commutative diagram
$$\xymatrix @R10pt @C10pt {0\ar[r] & N\otimes_R M \ar[r]^-{Id\otimes i} \ar[d] & N\otimes_RM' \ar[r]^-{Id\otimes p}  \ar[d] & N\otimes_RM''\ar[d] \\ 0\ar[r] &  N\otimes_R M \otimes_{\ZZ} R\ar[r]^-{Id\otimes i\otimes Id}  \ar@<1ex>[d] \ar@<-1ex>[d]  & N\otimes_R M' \otimes_{\ZZ} R\ar[r]^-{Id\otimes p\otimes Id}  \ar@<1ex>[d] \ar@<-1ex>[d] & N\otimes_RM''\otimes_{\ZZ} R \ar@<1ex>[d] \ar@<-1ex>[d] \\ 0\ar[r] & 
 N\otimes_R M \otimes_{\ZZ} R \otimes_{\ZZ} R\ar[r]^-{i'} & N\otimes_RM'\otimes_{\ZZ} R\otimes_{\ZZ} R  \ar[r]^-{p'}  &N\otimes_R M''\otimes_{\ZZ} R\otimes_{\ZZ} R
}$$ (where $i'=\Id\otimes i\otimes Id\otimes Id$ and $p'=\Id\otimes p\otimes Id\otimes Id$) whose rows are exact, as well as both the second and third columns, by Proposition \ref{super}. Hence, the first column is exact too.

\end{proof}

\medskip
\begin{proposition}\label{1.30}   Let $\mathbb F$ be a functor of $\mathcal R$-modules. Then, $${\mathbb F^\vee}^r(M)=\Hom_{\mathcal R}(\mathbb F,\mathcal M),$$
for any $R$-module $M$. Hence, ${\mathbb F^\vee}^{ro}=\mathbb F$.

\end{proposition}

\begin{proof} By Lema \ref{reperab} and Proposition \ref{super}, the sequence of morphisms 
$$\xymatrix @R6pt {S\otimes_R M\ar[r] &  S\otimes_R  R\langle M\rangle \ar@<1ex>[r]
\ar@<-1ex>[r]&  S\otimes_R R\langle M\oplus Rx\rangle\\ m \ar@{|->}[r] & m,\quad p(m) \ar@{|->}[r]<1ex>
\ar@{|->}[r]<-1ex>&
p(mx),\,p(m)x}$$
is exact for any $\,R$-algebra $S$.  that is, if $\mathcal M$, $\mathcal  R\langle M\rangle$ and $\mathcal R\langle M\oplus Rx\rangle$ are the quasi-coherent modules associated with $M,R\langle M\rangle$ and $R\langle M\oplus Rx\rangle$, respectively, then
the sequence of morphisms 
$$\xymatrix{\mathcal M\ar[r] &  \mathcal  R\langle M\rangle \ar@<1ex>[r]
\ar@<-1ex>[r]&  \mathcal  R\langle M\oplus Rx\rangle}$$
is exact. Hence,
${\mathbb F^\vee}^r(M)=\Hom_{\mathcal R}(\mathbb F,\mathcal M).$

\end{proof}
\medskip

\section{Adjoint functor theorem}

Given an $R$-algebra  $\,S\,$,  let $\,\pi_S\colon R\langle S\rangle \to S\,$ be the morphism of $\,R$-algebras $\,s\mapsto s\,$, for any $\,s\in S$.

\begin{definition} Let $\mathbb F$ be an $\R$-module. We have a natural morphism $\pi_{\mathbb F}\colon \mathbb F^{ro}\to \mathbb F$ defined as follows $$\pi_{\mathbb F,S}(m)=\mathbb F(\pi_S)(m),\,\, \text{for any } m\in \mathbb F^{ro}(S)=\mathbb F^{r}(S)\subseteq \mathbb F(R\langle S\rangle).$$
\end{definition}

\begin{proposition}  For any $s\in S$ and  $m\in \mathbb F^{ro}(S)=\mathbb F^{r}(S)\subseteq \mathbb F(R\langle S\rangle)$, 
$$\pi_{\mathbb F,S}(s* m)=s\cdot \mathbb F(\pi_S)(m).$$
(Recall Remark \ref{recall}).
\end{proposition}

\begin{proof} Given $m\in {\mathbb F}^r(S)$, we know that $\mathbb F(h_x)(m)-x\cdot \mathbb F(in)(m)=0$, by Definition \ref{D4.2}. Let $h_s\colon R\langle S\rangle\to R\langle S\rangle$ be defined by
$h_s(s')=s\cdot s'\in S\cdot S\subseteq R\langle S\rangle$.
Consider the morphism of $R$-algebras $R\langle S\oplus Rx\rangle \overset{x=s}\to R\langle S\rangle$, $s'\mapsto s'$ and $x\mapsto s$. We have the commutative diagrams
$$\xymatrix{R\langle S\rangle \ar[rd]^-{h_s} \ar[r]^-{h_x} & R\langle S\oplus xR\rangle  \ar[d]^-{x=s }\\ &  R\langle S\rangle }\quad 
\xymatrix{R\langle S\rangle \ar[r]^-{in}  \ar[dr]^-{Id} & R\langle S\oplus xR\rangle \ar[d]^-{x=s }\\  & R\langle S\rangle }
$$
Then, 
$$0  =\mathbb F(x=s)(\mathbb F(h_x)(m)-x\cdot \mathbb F(in)(m))=\mathbb F(h_s)(m)-s\cdot m.$$
Observe that $\pi_S\circ h_s=\pi_S\circ R\langle s\cdot \rangle$. Then,
$$\aligned 0 & =\mathbb F(\pi_S)(\mathbb F(h_s)(m)-s\cdot m)=\mathbb F(\pi_S)( \mathbb F(R\langle s\cdot \rangle)(m)-s\cdot m)\\  &  =\mathbb F(\pi_S)( s*m-s\cdot m)=
\mathbb F(\pi_S)( s*m)-s\cdot \mathbb F(\pi_S)(m)\\  & =\pi_{\mathbb F,S}(s*m)-s\cdot \pi_{\mathbb F,S}(m).\endaligned
$$

\end{proof}

\begin{proposition} \label{PB} Let $\phi\colon \mathbb F\to \mathbb F' $ be a morphism of $\R$-modules. The diagram
$$\xymatrix{\mathbb F^{ro} \ar[r]^-{\phi^{ro}} \ar[d]^-{\pi_{\mathbb F}} & \mathbb F'^{ro} \ar[d]^-{\pi_{\mathbb F'}} \\ \mathbb F \ar[r]_-\phi & \mathbb F'}$$
is commutative.
\end{proposition}

\begin{proof} The diagram
$$\xymatrix{\mathbb F^{ro}(S) \ar[r]^-{\phi^{ro}_S} \ar@/^-2.0pc/[dd]_-{\pi_{\mathbb F,S}} \ar@{^{(}->}[d] & \mathbb F'^{ro}(S) \ar@/^2.0pc/[dd]^-{\pi_{\mathbb F',S}} \ar@{^{(}->}[d]
\\ \mathbb F (R\langle S\rangle) \ar[r]_-{\phi_{R\langle S\rangle}} \ar[d]^-{\mathbb F(\pi_S)} & \mathbb F'(R\langle S\rangle ) \ar[d]_-{\mathbb F'(\pi_S)}
\\ \mathbb F (S) \ar[r]_-{\phi_S} & \mathbb F'(S)
}$$
is commutative.
\end{proof}

Given a right $R$-module  $\,N$,  let $\, i_N\colon N\to R\langle N\rangle \,$ be the morphism of $\,R$-modules $\,n\mapsto n\,$, for any $\,n\in N$.

\begin{definition} Let $\mathbb G$ be a functor of abelian groups. We have a natural morphism $i_{\mathbb G}\colon \mathbb G\to \mathbb G^{or}$ defined as follows:
$$i_{\mathbb G,N}(g):=\mathbb G(i_N)(g),\text{ for any } g\in \mathbb G(N)$$
Let us check that $\mathbb G(i_N)(g)\in \mathbb G^{or}(N)\subset \mathbb G^o(R\langle N\rangle)=\mathbb G(R\langle N\rangle)$: The composite morphism $$N\overset{i_N}\to R\langle N\rangle \overset{h_x-x\cdot in}\longrightarrow R\langle N\oplus xR\rangle$$
is zero. Hence, $$\aligned 0 & =\mathbb G(h_x-x\cdot in)(\mathbb G(i_N)(g))=
(\mathbb G(h_x)- \mathbb G(x\cdot in))(\mathbb G(i_N)(g))
\\ & = (\mathbb G^o(h_x)-x\cdot \mathbb G^o(in))(\mathbb G(i_N)(g))\endaligned$$
and $\mathbb G(i_N)(g)\in \mathbb G^{or}(N)$.
\end{definition}

\begin{proposition} \label{PA} Let $\phi\colon \mathbb G\to \mathbb G'$ be a morphism of functors of groups. The diagram
$$\xymatrix{\mathbb G\ar[r]^-\phi \ar[d]^-{i_{\mathbb G}} & \mathbb G' \ar[d]^-{i_{\mathbb G'}} \\ \mathbb G^{or} \ar[r]_-{\phi^{or}} & \mathbb G'^{or}}$$
is commutative.
\end{proposition}

\begin{proof} The diagram 
$$\xymatrix @R=10pt {\mathbb G(N) \ar[r]^-{\phi_N} \ar[d]^-{i_{\mathbb G,N}} & \mathbb G'(N) \ar[d]^-{i_{\mathbb G',N}} \\ \mathbb G^{or}(N) \ar@{-->}[r]_-{\phi^{or}_N} \ar@{^{(}->}[d] & \mathbb G'^{or}(N) \ar@{^{(}->}[d] \\
 \mathbb G^r(R\langle N\rangle) \ar[r]^-{\phi^r_{R\langle N\rangle}}  \ar@{=}[d] & \mathbb G'^r(R\langle N\rangle) \ar@{=}[d]
\\ \mathbb G(R\langle N\rangle) \ar[r]^-{\phi_{R\langle N\rangle}}  & \mathbb G'(R\langle N\rangle)}
$$
is commutative.

\end{proof}

\begin{lemma} \label{T2}  Let $\mathbb G$ be a functor of abelian groups.
The composite morphism
$$\mathbb G^o\overset{i_{\mathbb G}^o}\longrightarrow \mathbb G^{oro} \overset{\pi_{\mathbb G^o}}\longrightarrow \mathbb G^o$$
is the identity morphism.
\end{lemma}

\begin{proof} The diagram
$$\xymatrix{\mathbb G^o(S) \ar@{=}[d] \ar[r]^-{i^o_{\mathbb G,S}} &
\mathbb G^{oro}(S) \ar@{=}[d] \ar[r]^-{\pi_{\mathbb G^o,S}} &
\mathbb G^o(S) \ar@{=}[r] & \mathbb G(S)\\
\mathbb G(S) \ar[r]^-{i_{\mathbb G,S}}  \ar@/^-2.0pc/[rrr]_-{\mathbb G(i_S)}&
\mathbb G^{or}(S)  \ar@{^{(}->}[r] &
\mathbb G^o(R\langle S\rangle) \ar@{=}[r] \ar[u]^-{\mathbb G^o(\pi_S)} & \mathbb G(R\langle S\rangle) \ar[u]^-{\mathbb G(\pi_S)}
}$$
is commutative. Hence, $\pi_{\mathbb G^o,S}\circ i_{\mathbb G,S}^o=
\mathbb G(\pi_S)\circ \mathbb G(i_S)=\mathbb G(\pi_S\circ i_S)=Id$.
\end{proof}

\begin{lemma} \label{T1} Let $\mathbb F$ be an $\R$-module. The composite morphism
$$\mathbb F^r\overset{i_{\mathbb F^r}}\longrightarrow \mathbb F^{ror} \overset{\pi_{\mathbb F}^r} \longrightarrow \mathbb F^r$$
is the identity morphism.
\end{lemma}

\begin{proof} The diagram
$$\xymatrix{\mathbb F^r(N) \ar[d]_-{i_{\mathbb F^r,N}} \ar@{^{(}->}[rrr]  \ar[rrd]^-{\mathbb F^r(i_N)} & & & \mathbb F(R\langle N\rangle) \ar[d]^-{\mathbb F({R\langle i_N\rangle})} \\
\mathbb F^{ror}(N) \ar@{^{(}->}[r] \ar[d]_-{\pi^r_{\mathbb F,N}} &
\mathbb F^{ro}(R\langle N\rangle) \ar@{=}[r] \ar[d]_-{\pi_{\mathbb F,R\langle N\rangle}} & \mathbb F^{r}(R\langle N\rangle)
\ar@{^{(}->}[r] & \mathbb F(R\langle R\langle N\rangle\rangle)
\ar[dll]^-{\mathbb F(\pi_{R\langle N\rangle})}\\ \mathbb F^{r}(N) 
\ar@{^{(}->}[r] & \mathbb F(R\langle N\rangle)
}$$
is commutative. Then, $\pi^r_{\mathbb F,N}\circ i_{\mathbb F^r,N}=Id$
since $$\mathbb F(\pi_{R\langle N\rangle})\circ \mathbb F(i_{R\langle N\rangle})=\mathbb F(\pi_{R\langle N\rangle}\circ {R\langle i_N\rangle})=\mathbb F(Id)=Id.$$ Hence, $\pi^r_{\mathbb F}\circ i_{\mathbb F^r}=Id$.

\end{proof}

\begin{theorem} \label{T3} Let $\mathbb F$ be an $\R$-module and $\mathbb G$ an additive functor of abelian groups.  Then, we have the functorial isomorphism
$$\xymatrix @R=8pt { \Hom_{grp}(\mathbb G,\mathbb F^r) \ar@{=}[r] & 
\Hom_{\R}(\mathbb G^o,\mathbb F)\\ \phi \ar@{|->}[r]  & \pi_{\mathbb F}\circ\phi^o\\\varphi^r\circ i_{\mathbb G} & \varphi \ar@{|->}[l]}$$
\end{theorem}

\begin{proof} It is a check:
$$\aligned & \phi \mapsto \pi_{\mathbb F}\circ \phi^o\mapsto \pi^r_{\mathbb F}\circ \phi^{or}\circ i_{\mathbb G} \overset{\text{\ref{PA}}}= \pi^r_{\mathbb F}\circ i_{\mathbb F^r}\circ \phi\overset{\text{\ref{T1}}}=\phi\\
& \varphi \mapsto \varphi^r\circ i_{\mathbb G}\mapsto \pi_{\mathbb F}\circ \varphi^{ro}\circ i^o_{\mathbb G} \overset{\text{\ref{PB}}}
=\varphi\circ \pi_{\mathbb G^o}\circ i^o_{\mathbb G}\overset{\text{\ref{T2}}}=\varphi\endaligned$$
\end{proof}

\begin{corollary} \label{L5.111} Let $\,\mathbb F\,$ be an $\,\mathcal R$-module such that 
$\pi_{\mathbb F}\colon \mathbb F^{ro}\to \mathbb F$ is an isomorphism. 
If $\,\mathbb F'\,$ is another $\,\mathcal{R}$-module, the natural morphism
$$\Hom_{\mathcal R}(\mathbb F,\mathbb F')\to \Hom_{grp}({\mathbb F}^r , {\mathbb F'}^r),\, f\mapsto f^r.$$
\end{corollary}

\begin{proof} $\Hom_{\mathcal R}(\mathbb F,\mathbb F')
=
\Hom_{\mathcal R}({\mathbb F}^{ro} , {\mathbb F'})
\overset{\text{\ref{T3}}}=\Hom_{grp}({\mathbb F}^r , {\mathbb F'}^r)$.
\end{proof}
\medskip

\section{Reflexivity theorem}\label{lafinal}

Let $\,M\,$ be an $\,R-$module. The functor $\,{\mathcal M^\vee}^r\,$ is precisely the functor of (co)points of $\,M\,$ in the category of $\,R-$modules: if $\,Q\,$ is another $\,R-$module, in virtue of Proposition \ref{1.30}:
$${\mathcal{M}^\vee}^r(Q) = \Hom_{\mathcal{R}} (\mathcal{M} , \mathcal{Q}) = \Hom_{R}(M , Q) \ . $$

\begin{lemma} \label{Yoneda} Let $\,M\,$ be a right $\,R$-module and $\,\mathbb G\,$ an aditive functor of abelian groups. 
Then, $${\Hom}_{grp} (\mathcal M^{\vee r}, \mathbb G)=\mathbb G(M)\, .$$
\end{lemma}

\begin{proof} It is Yoneda's lemma.

\end{proof}

\begin{theorem}\label{prop4} Let $\,M\,$ be an $\,R$-module and $\,N\,$ be a right $\,R$-module. 
Then, $${\Hom}_{\mathcal R} ({\mathcal M^\vee}, {\mathcal N})={\mathcal N}^r(M) \, .$$
\end{theorem}

\begin{proof} $\,\mathcal{M}^\vee\,$ satisfies the hypothesis of Theorem \ref{L5.111} (see Proposition \ref{1.30}), so that
$$ {\Hom}_{\mathcal R} ({\mathcal M^\vee}, {\mathcal N}) \overset{\text{\ref{L5.111}}}= {\Hom}_{grp} ({{\mathcal M^\vee}^r}, {{\mathcal N}^r})\overset{\text{\ref{Yoneda}}}= {\mathcal N}^r(M) \ .
$$

\end{proof}

\begin{theorem} \label{reflex2}
Let $\,M\,$ be an $\,R$-module. The natural morphism of $\,\mathcal R$-modules $$\mathcal M \longrightarrow \mathcal M^{\vee\vee}$$ is an isomorphism.
\end{theorem}

\begin{proof} It is a consequence of Theorem \ref{prop4} and Proposition \ref{super}: 
$$\mathcal M^{\vee\vee}(S)=\Hom_{\mathcal R}(\mathcal M^\vee,\mathcal S)= \mathcal M^r(S)= S\otimes_R M=\mathcal M(S) \ .$$ 
\end{proof}

\medskip
\begin{theorem} \label{reflex}
Let $\,M\,$ be an $\,R$-module. 
The natural morphism of $\,\mathcal R$-modules $$\mathcal M \longrightarrow \mathcal M^{**}$$ is an isomorphism.
\end{theorem}

\begin{proof} Let $S$ be an $R$-algebra. $\,\mathcal M_{|S}\,$ is the $\mathcal S$-quasi-coherent module associated with $\,S\otimes_R M\,$ and $\,{\mathcal M_{|S}}^{*}= {\mathcal M_{|S}}^{\vee}\,$. Then, 

$$\aligned \mathcal M^{**}(S) & =\Hom_{\mathcal S}({\mathcal M^*}_{|S},\mathcal S)=
\Hom_{\mathcal S}({\mathcal M_{|S}}^*,\mathcal S)\overset{\text{\ref{prop4}}}={\mathcal M_{|S}}^r(S) \\ & \overset{\text{\ref{super}}}=
S\otimes_S (S\otimes_R M) =S\otimes_R M=\mathcal M(S). \endaligned $$

\end{proof}

\section{Quasi-coherent modules associated with flat modules}

Given an $R$-module $M$, let $\mathcal M_r$ be the functor of abelian groups defined by
$$\mathcal M_r(N):=N\otimes_R M.$$
for any  right $R$-module $N$. Observe that 
there exists a natural morphism $\mathcal M_r\to \mathcal M^r$ and 
$\mathcal M_r^{\, o}=\mathcal M^{ro}=\mathcal M$, by Proposition \ref{super}.

In \cite{mittagleffler}, it is given several characterizations of $\mathcal M_r$, when $M$ is a flat or projective  module. 
The adjoint functor theorem will give us the corresponding characterizations of $\mathcal M$, when $M$ is a flat or projective module. 
We have only to add the following hypothesis.

\begin{hypothesis} \label{HP2} From now on we wil assume that $\mathcal M_r=\mathcal M^r$ (if $M$ is a flat $R$-module, then $\mathcal M_r=\mathcal M^r$, by Proposition \ref{super}.)
\end{hypothesis}

\begin{definition} We will say that $\mathcal N^\vee$ is the module scheme associated with $N$.
\end{definition}

\begin{theorem} $M$ is a finitely generated projective module iff $\mathcal M$ is a module scheme.\end{theorem}

\begin{proof} $\Rightarrow)$ By \cite[3.1]{mittagleffler}, there exists an isomorphism
$\mathcal M_r\simeq \mathcal N^{\vee r}$. Then,
$$\mathcal M=\mathcal M_r^{\, o}\simeq \mathcal N^{\vee ro}=\mathcal N^{\vee}.$$

$\Leftarrow)$ 
$\mathcal M\simeq \mathcal N^\vee$. Then, $\mathcal M_r=\mathcal M^r \simeq \mathcal N^{\vee r}$. By \cite[3.1]{mittagleffler}, $M$ is a finitely generated projective module.

\end{proof}

\begin{theorem} $M$ is a flat module iff $\mathcal M$ is a direct limit of module schemes.\end{theorem}

\begin{proof} $\Rightarrow)$ By \cite[3.2]{mittagleffler}, there exists an isomorphism
$\mathcal M_r\simeq \ilim{i\in I} \mathcal N_i^{\vee r}$. Then,
$$\mathcal M=\mathcal M_r^{\,s}\simeq  (\ilim{i\in I} \mathcal N_i^{\vee r})^o= \ilim{i\in I} \mathcal N_i^{\vee ro}=\ilim{i\in I} \mathcal N_i^{\vee}.$$

$\Leftarrow)$ 
$\mathcal M\simeq \ilim{i\in I} \mathcal N^\vee$. Then, $\mathcal M_r=\mathcal M^r \simeq \ilim{i\in I} \mathcal N^{\vee r}$. By \cite[3.2]{mittagleffler}, $M$ is flat.

\end{proof}

\begin{lemma} \label{L6.5A} Let $N_1,N_2$ be right $R$-modules. Then, \begin{enumerate}

\item  $ \Hom_{\R}(\mathcal N_1^\vee,\mathcal  N_2^\vee) =\Hom_{grp}(\mathcal N_1^{\vee r},\mathcal \mathcal  N_2^{\vee r})$.

\item  $ \Hom_{\R}(\mathcal N_1,\mathcal  N_2) =\Hom_{grp}(\mathcal N_{1 r},\mathcal \mathcal  N_{2 r})$.

\item  $\Hom_{\R}(\mathcal N_1^\vee,\mathcal M) =\Hom_{grp}(\mathcal N_1^{\vee r},\mathcal M_r)$.

\item  $\Hom_{\R}(\mathcal M^\vee,\mathcal N_1) =\Hom_{grp}(\mathcal M^{\vee r},\mathcal N_{1r})$.
\end{enumerate} 
\end{lemma} 

\begin{proof} 1. It is Corollary \ref{L5.111}.

2. $ \Hom_{\R}(\mathcal N_1,\mathcal  N_2) \overset{\text{\ref{tercer}}}=\Hom_{R}(N_1,N_2)\overset{\text{\cite[2.4]{mittagleffler}}}= \Hom_{grp}(\mathcal N_{1 r},\mathcal \mathcal  N_{2 r})$.

3. $\Hom_{\R}(\mathcal N_1^\vee,\mathcal M) \overset{\text{\ref{L5.111}}}=\Hom_{grp}(\mathcal N_1^{\vee r},\mathcal M^r)=\Hom_{grp}(\mathcal N_1^{\vee r},\mathcal M_r)$.

4. $\Hom_{\R}(\mathcal M^\vee,\mathcal N_1) \overset{\text{\ref{prop4}}}=\mathcal N_1^r(M)\overset{\text{\ref{repera2b}}}=\mathcal M^r(N_1)=\mathcal M_r(N_1)=N_1\otimes_RM$\newline \phantom{pp}
\hskip 3cm $\overset{\text{\ref{Yoneda}}}=\Hom_{grp}(\mathcal M^{\vee r},\mathcal N_{1r})$.

\end{proof}

\begin{lemma} \label{L6.5} 
Let $f\colon  \mathbb F_1\to \mathbb F_2$ be a morphism of $\R$-modules. Then, $f$ is a monomorphism
iff the morphism of functors of groups $f^r\colon  \mathbb F_1^r\to \mathbb F_2^r$ is a monomorphism.

\end{lemma} 

\begin{proof} If $f$ is a monomorphism, then $f^r$ is a monomorphism, since $f^r_N=f_{R\langle N\rangle}$ on $\mathbb F^r_1(N)\subseteq 
\mathbb F_1(R\langle N\rangle)$. If $f^r$ is a monomorphism, then $f=f^{ro}$
is a monomorphism since $f^{ro}_{S}=f^r_S$.

\end{proof}

\begin{theorem} \label{ML} Let $M$ be an $R$-module. The following statements are equivalent

\begin{enumerate}
\item $M$ is a flat Mittag-Leffler module.

\item Every morphism of $\mathcal R$-modules $\mathcal N^\vee\to \mathcal M$ factors through an $\R$-submodule scheme of $\mathcal M$, for any right $R$-module $N$.

\item $\mathcal M$ is equal to a direct limit of  $\R$-submodule schemes.
\end{enumerate}
\end{theorem}

\begin{proof} $1. \iff 2.$ It is an immediate consequence of \cite[4.5]{mittagleffler}, Lemma \ref{L6.5A}  and Lemma \ref{L6.5}.

$1. \iff 3.$ It is an immediate consequence of \cite[4.5]{mittagleffler}  and Lemma \ref{L6.5}, since $\mathcal M\simeq \ilim{i\in I}\mathcal N_i^\vee$ iff $\mathcal M_r=\mathcal M^r\simeq \ilim{i\in I}\mathcal N_i^{\vee r}$. 

\end{proof}

\begin{lemma} \label{L6.8} A morphism of $\R$-modules $f\colon \mathcal M^{\vee} \to \prod_{i\in I} \mathcal N_i$ is an epimorphism iff the corresponding morphism of functors of groups  $\mathcal M^{\vee r}\to  \prod_{i\in I} \mathcal N_{ir}$ (see \ref{L6.5A} (4))  is an epimorphism.\end{lemma}

\begin{proof} $\Rightarrow)$ $f=(\sum_j n_{ij}\otimes m_{ij})_{i\in I}$ through the equality
$$\Hom_{\R}(\mathcal M^{\vee}, \prod_{i\in I} \mathcal N_i)=
\Hom_{\R}(\mathcal M^{\vee r}, \prod_{i\in I} \mathcal N_{i,r})\overset{\text{\ref{Yoneda}}}=
\prod_{i\in I} \mathcal N_{i,r}(M)=\prod_{i\in I} ( N_i\otimes_R M).$$

We have to prove that the morphism 
$$\xymatrix{\Hom_{R}(M,N) \ar@{=}[r] & \mathcal M^{\vee r}(N)\ar[r] &  \prod_{i\in I} \mathcal N_{ir}(N)\ar@{=}[r] & \prod_{i\in I} (N_{ir}\otimes N)\\  h \ar@{|->}[rrr] & & & (\sum_j n_{ij}\otimes h(m_{ij}))_{i\in I}}$$ is an epimorphism, for any $R$-module $N$. If $N$ is an $R$-algebra, then it is an epimorphism, since $f$ is an epimorphism.
We can suppose that $N$ is a free $R$-module, since
the functor $\prod_{i\in I} \mathcal N_{ir}$ preserves epimorphisms.
In this case $N$ is naturally a bimodule. Let $\pi\colon R\langle N\rangle \to N$ be the composition of the obvious morphisms of $R$-modules $R\langle N\rangle\to N\cdot R\to N$. Obviously, $\pi$ is an epimorphism. Then, we can suppose that $N$ is an $R$-algebra. We conclude.
\end{proof}

\begin{theorem} \label{blabla} Let $\,M\,$ be an $\,R$-module. The following statements are equivalent:

\begin{enumerate} 
\item $\,M\,$ is a flat strict Mittag-Leffler module.

\item Any morphism $\,f\colon \mathcal M^\vee\to \mathcal N\,$ factors through the quasi-coherent module associated with $\,\Ima f_R$, for any right $\,R$-module $\,N$.

\item Any morphism $\,f\colon \mathcal M^\vee\to \mathcal R\,$ factors through the quasi-coherent module associated with $\,\Ima f_R$.

\item Let $\,\{M_i\}_{i\in I}\,$ be the set of all  finitely generated $\,R$-submodules of $\,M$, and $\,M'_i:=\Ima[M^*\to M_i^*]$. The natural morphism
$$\mathcal M\to \lim \limits_{\rightarrow} {\mathcal M_i'}^\vee$$
is an isomorphism.

\item $\,\mathcal M\,$ is a direct limit of submodule schemes, $\,\mathcal N_i^\vee\subseteq\mathcal M\,$ and the dual morphism
$\,\mathcal M^\vee\to \mathcal N_i\,$ is an epimorphism, for any $\,i$.

\item There exists a monomorphism $\,\mathcal M\hookrightarrow \prod_{I}\mathcal R$.

\end{enumerate}
\end{theorem}

\begin{proof} $ 1. \iff 2.$ It is a consequence of 
\cite[4.9\,(2)]{mittagleffler} and Lemma \ref{L6.5A}\,4. 

$ 1. \iff 3.$ It is a consequence of 
\cite[4.10]{mittagleffler} and Lemma \ref{L6.5A}\,4. 

$ 1. \iff 4.$ It is a consequence of 
\cite[4.9\,(3)]{mittagleffler}. 

$ 1. \iff 5.$ It is a consequence of 
\cite[4.9\,(4)]{mittagleffler} and Lemma \ref{L6.8}. 

$ 1. \iff 6.$ It is a consequence of 
\cite[4.7\,(3)]{mittagleffler} and Lemma \ref{L6.5}. 

\end{proof}

\begin{proposition} \label{CR7} An $R$-module $M$ is a projective $R$-module of countable type if and only if there exists a chain of submodule schemes of $\mathcal M$
$$\mathcal N_1^\vee\subseteq \mathcal N_2^\vee\subseteq \cdots\subseteq \mathcal N_n^\vee\subseteq \cdots$$
such that $\mathcal M=\cup_{n\in\mathbb N} \mathcal N_n^\vee$. \end{proposition}

\begin{proof} It is a consequence of 
\cite[4.11,13]{mittagleffler} and Lemma \ref{L6.5}. 
\end{proof}

 \begin{theorem} An $R$-module $M$ is projective if and only if there exists a chain of $\R$-sub\-mo\-du\-les of $\mathcal M$
 $$\mathbb W_1\subseteq \mathbb W_2\subseteq \cdots\subseteq \mathbb W_n\subseteq \cdots$$
such that $\mathcal M=\cup_{n\in\mathbb N} \mathbb W_n$, where $\mathbb W_n$ is a direct sum of module schemes  and the natural morphism $\mathcal M^\vee\to \mathbb W_{n}^\vee$ is an epimorphism, for any $n\in \mathbb N$.
 \end{theorem}

\begin{proof} It is a consequence of 
\cite[4.14]{mittagleffler} and Lemma \ref{L6.8}. 
\end{proof}


\begin{thebibliography}{99}


\bibitem{Amel} \textsc{\'Alvarez, A., Sancho, C., Sancho, P.:}
\textit{Algebra schemes and their representations}, J. Algebra
{\bf 296/1} (2006) 110--144.


\bibitem{Pedro1} \textsc{\'Alvarez, A., Sancho, C., Sancho, P.:}
\textit{A characterization of linearly semisimple groups},
Acta. Math. Sin.-English Ser. {\bf 27/1} (2011) 185--192.

\bibitem{Pedro2} \textsc{\'Alvarez, A., Sancho, C., Sancho, P.:}
\textit{\! Reynolds operator on functors},
J. Pure Appl. Algebr. {\bf 215/8} (2011) 1958--1966.

\bibitem{gabriel} \textsc{Demazure, M., Gabriel, P.:}
\textit{\! Introduction to Algebraic Geometry and Algebraic
Groups}, Mathematics Studies {\bf 39}, North-Holland, 1980.

\bibitem{mittagleffler} \textsc{Gordillo-Merino, A., Navarro, J., Sancho, P.:} \textit{\! Functors of modules associated with flat and projective modules}, \texttt{arXiv:1710.04153v3}. 

\bibitem{Hirschowitz} \textsc{Hirschowitz, A.:} \textit{Coh\'{e}rence et dualit\'{e} sur le gros site de Zariski}, Algebraic curves and projective geometry (Trento, 1988), LNM, 1389, Springer, Berlin, 1989.







\end{thebibliography}
\end{document}